\documentclass[reqno]{amsart}

\usepackage{tikz}
\usepackage{tikz-cd}
\usepackage{amsmath}
\usepackage{amsfonts}
\usepackage{amssymb,enumerate}
\usepackage{amsthm}
\usepackage[all]{xy}
\usepackage{rotating}
\usepackage{amsmath,amssymb,amsthm,amsfonts,amsxtra,mathtools}
\usepackage{enumerate,verbatim,mathrsfs,comment,color,url}
\usepackage[pagebackref]{hyperref}
\usepackage[margin=1in]{geometry}
\usepackage{xcolor}

\usepackage{soul}
\newif\ifshowcomments
\showcommentstrue 
\showcommentsfalse 

\theoremstyle{plain}
\newtheorem{lem}{Lemma}[section]
\newtheorem{cor}[lem]{Corollary}
\newtheorem{prop}[lem]{Proposition}
\newtheorem{thm}[lem]{Theorem}

\newtheorem{intthm}{Theorem}

\newcommand{\isom}{\cong}

\theoremstyle{definition}

\newtheorem{discussion}[lem]{Discussion}
\newtheorem{ex}[lem]{Example}
\newtheorem{question}[lem]{Question}

\newtheorem{para}[lem]{}

\newtheorem{convention}[lem]{Convention}
\newtheorem*{convention*}{Convention}

\newcommand{\syz}{\Omega}

\def \pd {\operatorname{pd}}

\def \im {\operatorname{Im}}
\def \ord {\operatorname{ord}}

\def \B {\operatorname{Burch}}
\def \edim {\operatorname{edim}}

\def \soc {\operatorname{soc}}

\def \del {\partial}

\newcommand{\depth}{\operatorname{depth}}

\newcommand{\ann}{\operatorname{Ann}}

\newcommand{\type}{\operatorname{type}}

\newcommand{\gr}{\operatorname{gr}}

\newcommand{\coker}{\operatorname{Coker}}

\newcommand{\End}{\operatorname{End}}

\newcommand{\Ker}{\operatorname{Ker}}

\newcommand{\ideal}[1]{\mathfrak{#1}}
\newcommand{\m}{\ideal{m}}
\newcommand{\n}{\ideal{n}}

\newcommand{\fm}{\ideal{m}}
\newcommand{\fn}{\ideal{n}}

\newcommand{\xra}{\xrightarrow}
\newcommand{\xla}{\xleftarrow}

\def\syz{\mathrm{\Omega}}

\renewcommand{\geq}{\geqslant}
\renewcommand{\leq}{\leqslant}
\renewcommand{\ker}{\Ker}

\newcommand{\Ext}[4][R]{\operatorname{Ext}_{#1}^{#2}(#3,#4)}

\newcommand{\Hom}{\operatorname{Hom}}	
\newcommand{\Tor}[4][R]{\operatorname{Tor}^{#1}_{#2}(#3,#4)}

\def\Tor{\operatorname{Tor}}
\def\Ext{\operatorname{Ext}}
\def\tr{\operatorname{tr}}
\newcommand{\lo}{\longrightarrow}
\def\Tr{\operatorname{\textsf{Tr}}}

\def \Tr {\operatorname{Tr}}
\def \ord {\operatorname{ord}}

\numberwithin{equation}{lem}
\def \del {\partial}

\begin{document}

\bibliographystyle{amsplain}

\title[Vector space summands of lower syzygies]{Vector space summands of lower syzygies}

\author[M. Asgharzadeh]{Mohsen Asgharzadeh}
\address{Hakimiyeh\\
Tehran, Iran}
\email{mohsenasgharzadeh@gmail.com}

\author[M. DeBellevue]{Michael DeBellevue}
\address{Mathematics Department, Syracuse University, Syracuse, NY 13244 U.S.A.}
\email{mpdebell@syr.edu}

\author[S. Dey]{Souvik Dey}
\address{Faculty of Mathematics and Physics, Department of Algebra, Charles University, Sokolovska 83, 186 75 Praha, Czech Republic}
\email{souvik.dey@matfyz.cuni.cz}

\author[S. Nasseh]{Saeed Nasseh}
\address{Department of Mathematical Sciences\\
Georgia Southern University\\
Statesboro, GA 30460, U.S.A.}
\email{snasseh@georgiasouthern.edu}

\author[R. Takahashi]{Ryo Takahashi}
\address{Graduate School of Mathematics\\
Nagoya University\\
Furocho, Chikusaku, Nagoya, Aichi 464-8602, Japan}
\email{takahashi@math.nagoya-u.ac.jp}

\thanks{Souvik Dey was partly supported by Charles University Research Center program No. UNCE/24/SCI/022 and a grant GACR 23-05148S from the Czech Science Foundation. Ryo Takahashi was partly supported by JSPS Grant-in-Aid for Scientific Research 23K03070.}




\begin{abstract}
In this paper, we investigate problems concerning when the residue field $k$ of a local ring $(R,\m,k)$ appears as a direct summand of syzygy modules. First, we prove that the following conditions are equivalent: (i) $k$ is a direct summand of second syzygies of all non-free finitely generated $R$-modules; (ii) $k$ is a direct summand of third syzygies of all non-free finitely generated $R$-modules; (iii) $k$ is a direct summand of $\m$. We also prove various consequences of these conditions.
Next, we investigate for what artinian local rings $R$ the dual $E^*=\Hom_R(E_R(k),R)$ of the injective envelope of the residue field $k$, which is also a second syzygy, is annihalated by the maximal ideal. 
\end{abstract}

\maketitle

\section{Introduction}\label{section120215a}

\begin{convention}\label{conv20240916a}
Throughout the paper, $(R,\fm,k)$ is a commutative noetherian local ring which is not a field and all modules are finitely generated. The Krull dimension, depth, embedding dimension, Loewy length, multiplicity, and type of $R$ are denoted by $\dim(R)$, $\depth(R)$, $\edim(R)$, $\ell\ell(R)$, $e(R)$, and $\type(R)$, respectively. For an $R$-module $M$, we set $M^*=\Hom_R(M,R)$ and for a positive integer $n$, the notations $\syz^n(M)$ and $M^{\oplus n}$ stand for the $n$-th syzygy in the minimal $R$-free resolution of $M$ and the direct sum of $n$ copies of $M$. Also, $\ell_R(M)$, $\soc(M)$, and $\lambda_R(M)$ are the length of $M$, socle of $M$, and the minimal number of generators of $M$, respectively. The injective envelope $E_R(k)$ of $k$ will be simply denoted by $E_R$ (or just $E$ when the ring in question is clear) and finally, the Matlis dual $\Hom_R(M,E)$ of an $R$-module $M$ is denoted by $M^{\vee}$.
\end{convention}

The notion of Burch rings was introduced and studied by Dao, Kobayashi, and Takahashi in~\cite{bur}, where among other results in this context, they proved that $R$ is Burch if and only if $k$ is a direct summand of $\syz^2(k)$; see~\cite[Theorem 4.1]{bur}. Later, Dao and Eisenbud~\cite{eda} introduced the notion of Burch index of a local ring $R$, denoted $\B(R)$, expanding the definition of Burch rings and proved the following result; see~\cite[Definition 2.2 and Theorem 2.3]{eda} for the definitions of these notions.

\begin{thm}[\protect{\cite[Theorem 1.1]{eda}}]\label{thm20240910a}
Let $M$ be a non-free $R$-module. If $\depth(R)=0$ and $\B(R)\geq 2$, then $k$ is a direct summand of $\syz^n(M)$ for some $n\in \{4, 5\}$ and for all integers $n\geq 7$. 
\end{thm}

It is shown by Dao and Eisenbud~\cite[Example 4.6]{eda} that the assertion in Theorem~\ref{thm20240910a} does not necessarily occur for $n=4$. In a recent paper, using techniques from differential graded homological algebra and utilizing the notion of the Bar resolution, DeBellevue and Miller~\cite[Theorem A]{DBM} improved the lower bound $7$ in Theorem~\ref{thm20240910a} by showing that, under the same assumptions, $k$ is a direct summand of $\syz^n(M)$ for all integers $n\geq 5$.\vspace{2mm}

Our work in this paper was initially motivated by the above-mentioned results of~\cite{eda} and~\cite{DBM}. However, other relevant problems involving the residue field $k$ emerged along the way, one of which will be stated below as Question~\ref{question20240911a}  and the rest have been listed throughout the paper. In this introduction, we will discuss our results on the residue field $k$ from two perspectives.
First, following the theme from~\cite{eda} and~\cite{DBM}, we investigate the situation where $k$ is a direct summand of \emph{lower} syzygy modules, i.e., when $n<5$ in Theorem~\ref{thm20240910a} and \cite[Theorem A]{DBM}. This is the subject of Sections~\ref{sec20240910a} and~\ref{sec20240912z}, which from this point of view, can be considered as an addendum to \emph{op. cit.} documenting several results that are included in the statements of Theorems~\ref{int12} and~\ref{i1} below.

\begin{intthm}\label{int12}
The following conditions are equivalent:
\begin{enumerate}[\rm(i)]
\item\label{int121}
$k$ is a direct summand of $\syz^2(M)$ for every non-free $R$-module $M$;   
\item\label{int122}
$k$ is a direct summand of $\syz^3(M)$ for every non-free $R$-module $M$;
\item\label{int124}
$k$ is a direct summand of $\syz^1(k)=\fm$;
\item\label{int123}
$\soc(R)\nsubseteq \m^2$.
\end{enumerate}
\end{intthm}

If any of the equivalent conditions~\eqref{int121}-\eqref{int123} in Theorem~\ref{int12} holds, then by the aforementioned result~\cite[Theorem 4.1]{bur} of Dao, Kobayashi, and Takahashi, $R$ is Burch with $\depth(R)=0$. Adding to this, in Theorem~\ref{i}, we will show that this same conclusion holds if we replace $\syz^2$ or $\syz^3$ by $\syz^4$. Parts (a) and (b) of the next result contribute even more in this direction, while part (c) addresses a problem posed by Ramras~\cite{ramg} regarding the reflexive modules; see~\ref{para20240915v} for the definition.

\begin{intthm}\label{i1}
Assume that one of the equivalent conditions~\eqref{int121}-\eqref{int123} in Theorem~\ref{int12} holds. Then:
\begin{enumerate}[\rm(a)]
\item
The equality $\B(R)=\edim(R)$ holds.
\item
If $R$ is Gorenstein, then it is an artinian principal ideal ring of minimal multiplicity.
\item
If $R$ is Cohen-Macaulay and non-Gorenstein and $M$ is a reflexive $R$-module, then $M$ is free.
\end{enumerate}
\end{intthm}

The proofs of Theorems~\ref{int12} and~\ref{i1} take up almost the entire Section~\ref{sec20240910a}. The conditions~\eqref{int121}-\eqref{int124} in Theorem~\ref{int12} are readily computable for particular rings and syzygy modules, which has allowed us to find many examples that do and do not satisfy these conditions.  These examples, as well as some involving higher syzygy modules, are collected in Section~\ref{sec20240912z}.\vspace{2mm}



The second point of this article is to investigate the following question regarding $E^*$ being  annihilated by the maximal ideal; this is the subject of Sections~\ref{sec20240911a} and \ref{sec20240911d}. We note here that this property is equivalent to saying $\tr_R(E)=\soc(R)$ (see \ref{tinytr}). Since the socle is contained in any trace ideal of an Artinian local ring, this condition is equivalent to the ring having tiny canonical trace. 

Notice that when $R$ is not Gorenstein, $E^*$ is the second syzygy of the Auslander transpose of $E$; see~\ref{para20240911c}. Hence, this question is closely related to our work in Sections~\ref{sec20240910a} and~\ref{sec20240912z}.

\begin{question}\label{question20240911a}
Assuming that $R$ is artinian, when is $E^*$ a $k$-vector space? In other words, for what artinian local rings $R$ does the equality $\fm E^*=0$ hold?
\end{question}

Assuming that $R$ is artinian, since $E^*=\Ext^0_R(E,R)$ and $E$ is the canonical module of $R$, the condition $\m E^*=0$ in Question~\ref{question20240911a} can be considered as a counterpart to the assumptions in the following question posed by Dao, Kobayashi, and Takahashi; see also~\cite[Question 1]{Lyle}.

\begin{question}[\protect{\cite[Question 4.3]{trcan}}]\label{question20241014a}
If a local ring $R$ is Cohen-Macaulay with canonical module $\omega_R$, does $\m \Ext^{i>0}_R(\omega_R,R)=0$ imply that $R$ is nearly Gorenstein?
\end{question}

Lyle and Maitra~\cite[Example 5.2]{Lyle} showed that the answer to Question~\ref{question20241014a} is negative in general, however, they provided an affirmative answer to this question when $R$ is a local Cohen-Macaulay numerical semigroup ring of minimal multiplicity; see~\cite[Theorem 4.10]{Lyle}. 
The main results of Section~\ref{sec20240911a} 
are  Proposition~\ref{g} and Theorem~\ref{42new}. 
In Section~\ref{sec20240911d} we prove the following theorem which classifies the artinian non-Gorenstein nearly Gorenstein local rings that satisfy the condition $\m E^*=0$ and also provides a necessary condition for $E^*$ to be a $k$-vector space. This result shows that, outside of an exceptional case, the condition $\m E^*=0$ implies that a ring is not nearly Gorenstein, and suggests more generally that the dimension of the largest $k$-summand of $E^*$ may be a good numerical measurement of the failure of a ring to be nearly Gorenstein.  A part of statement~\eqref{int13a} in this result was independently discovered by Lyle and Maitra; see~\cite[Proposition 4.1]{Lyle}.

\begin{intthm}\label{int13}
The following statements hold:
\begin{enumerate}[\rm(a)]
\item\label{int13a}
If $R$ is not Gorenstein, then $\fm^2=0$ if and only if $R$ is nearly Gorenstein and the equality $\m E^*=0$ holds.
\item\label{int13b}
Assume that $R$ is artinian and $\m E^*=0$. If $R\cong S/I$, where $(S,\n,k)$ is an artinian Gorenstein local ring and $I$ is an ideal of $S$, then $\n(0:_SI)\subseteq I$.

\item Let $R$ be the fiber product of non-Gorenstein Artinian local rings $(S,\fm_S,k)$ and $(T,\fm_T,k)$ along $k$. Then, $E_R^*$ is annihilated by $\fm$ if and only if $E_S^*$ and $E_T^*$ are annihilated by $\fm_S$ and $\fm_T$ respectively. 
\end{enumerate} 
\end{intthm}

In the final section, we present a series of observations regarding the 1-dimensional analogue  of Question~\ref{question20240911a}. The primary focus is on the following result:

\begin{intthm}\label{int13'}
	Let $(R,\fm)$ be a 1-dimensional Cohen-Macaulay local ring admitting a canonical module $\omega=\omega_R$ such that the associated graded module $\gr_\fm(\omega^*)$ has positive depth. If $R$ is not Gorenstein and $\fm^2$ is Ulrich, then $\fm \omega^*$ is Ulrich. 
\end{intthm}
Additionally, we demonstrate that the assumptions of  Theorem \ref{int13'} are sharp and provide examples to discuss the converse.

\section{Residue field being a direct summand of syzygy modules}\label{sec20240910a}

This section is entirely devoted to the proofs of Theorems~\ref{int12} and~\ref{i1} from the introduction. We start with the following discussion that will be used later.

\begin{para}\label{1i}
If $\depth(R)=0$, then we have $\soc(R) \cong k^{\oplus r}$ for some positive integer $r$. On the other hand, writing $\m=(x_1,\ldots,x_n)$, it follows from the exact sequence
$$
0\lo \soc(R)\lo R \xra{\begin{pmatrix}x_1&\cdots&x_n\end{pmatrix}^{tr}} R^{\oplus n}
$$
of $R$-modules that $\soc(R)\cong k^{\oplus r}\cong \syz^2(M)$ for some $R$-module $M$.
\end{para}

\begin{para}
As we mentioned in the introduction, the notion of Burch index of $R$, denoted $\B(R)$, was introduced by Dao and Eisenbud in~\cite{eda}. This notion generalizes the definition of Burch rings introduced by Dao, Kobayashi, and Takahashi in~\cite{bur}. In fact, a local ring $R$ is Burch if $\B(R)\geq 1$. We refer the reader to the above-mentioned papers for the precise definitions of Burch ring and Burch index. 
\end{para}



The following well-known statement will be used frequently in the subsequent results and examples.

\begin{lem}\label{HVS}
Let $M$ be a non-zero $R$-module. Then, $k$ is a direct summand of $M$ if and only if $\soc(M)\nsubseteq \fm M$. In particular, $k$ is a direct summand of $\m$ if and only if $\soc(R)\nsubseteq \fm^2$.
\end{lem}

In the proof of this lemma, for simplicity, we identify isomorphisms by equalities.

\begin{proof}
For the ``only if'' part, assume that $k$ is a direct summand of $M$ and write $M= k^{\oplus a}\oplus M'$, where $a$ is a positive integer such that $k$ is not a direct summand of the $R$-module $M'$. Then, $\soc(M)= k^{\oplus a}\oplus \soc(M')$ and $\m M=\m M'$. Now, let $x\in k^{\oplus a}\subseteq \soc(M)$ be a non-zero element. If $x\in \m M'\subseteq M'$, then $x\in k^{\oplus a}\cap M'=(0)$. This means that $x=0$, which is a contradiction. Thus, $x\notin \m M'=\m M$, and therefore, $\soc(M)\nsubseteq \fm M$.

For the ``if'' part, assume that $\soc(M)\nsubseteq \fm M$. This implies that $\soc(M)\neq (0)$. Let $z\in \soc(M)\setminus \m M$ be a non-zero element and note that $Rz=k$.
Consider a basis $G=\{\overline{x}_1=\overline{z},\ldots,\overline{x}_n\}$ for the $k$-vector space $M/\m M$, where each $x_i$ is an element of $M$, and extend this to a minimal generating set $\{x_1=z,\ldots,x_n,y_1,\ldots,y_m\}$ for $M$. Let $N$ be the $R$-submodule of $M$ generated by $\{x_2,\ldots,x_n,y_1,\ldots,y_m\}$ and note that $M=Rz+N$. Suppose $rz\in N$ for some $r\in R$. Hence, there are $r_2,\ldots,r_n,s_1,\ldots,s_m\in R$ such that $rz=r_2x_2+\cdots r_nx_n+s_1y_1+\cdots+s_my_m$. Moding out by $\m M$ and using the fact that $G$ is linearly independent, we conclude that $r\in \m$. Since $z\in \soc(M)$, we have $rz=0$. Thus, $M=Rz\oplus N$, which along with the fact that $Rz=k$, implies that $k$ is a direct summand of $M$. 

The last statement follows from the fact that $\soc(R)=\soc(\m)$.
\end{proof}

We now provide the proof of Theorem~\ref{int12}. Note that the implication \eqref{int123}$\implies$\eqref{int121} has been proven in~\cite[Proposition 3.3(2)]{ains}. However, we provide an alternative proof for the reader's convenience.\vspace{2mm}

\noindent \emph{Proof of Theorem~\ref{int12}.}
Note that the assumptions in conditions~\eqref{int121}-\eqref{int123} all imply that $\depth(R)=0$. Also, the equivalence~\eqref{int124}$\Longleftrightarrow$~\eqref{int123} is the latter part of Lemma~\ref{HVS}.\vspace{2mm}

\eqref{int121}$\implies$\eqref{int122}: Assume that $k$ is a direct summand of $\syz^2(M)$ for every non-free $R$-module $M$. Since $\depth(R)=0$, such an $R$-module $M$ has infinite projective dimension, i.e., none of the syzygies of $M$ are free. The assertion in~\eqref{int122} now follows from the isomorphism $\syz^3(M) \cong \syz^2(\syz^1(M))$ of $R$-modules.\vspace{2mm}

\eqref{int122}$\implies$\eqref{int123}: Since $\depth(R) = 0$, by~\ref{1i}, there exist a positive integer $r$ and a non-free $R$-module $M$ such that $k^{\oplus r} \cong \syz^2(M)$. Passing to the syzygies, we obtain 
\begin{equation}\label{eq20240912a}
\m^{\oplus r} \cong \syz^1(k)^{\oplus r} \cong \syz^3(M).
\end{equation}
It follows from~\eqref{int122} and~\eqref{eq20240912a} that $k$ is a direct summand of $\m^{\oplus r}$. Since $\End_R(k) \cong k$ is a local ring, using~\cite[Lemma 1.2]{lw} we conclude that $k$ is a direct summand of $\m$. Thus, by Lemma~\ref{HVS} we have $\soc(R) \nsubseteq \m^2$.\vspace{2mm}

\eqref{int123}$\implies$\eqref{int121}: Assume that $\soc(R) \nsubseteq \m^2$. Let $M$ be a non-free $R$-module and suppose on the contrary that $k$ is not a direct summand of $\syz^2(M)$. By Lemma~\ref{HVS},
\begin{equation}\label{eq20240912c}
\soc(\syz^2(M))\subseteq \m \syz^2(M).
\end{equation}
Now consider the exact sequence
\begin{equation}\label{eq20240912b}
0\to \syz^2(M) \xra{i} R^{\oplus a}\xra{f} R^{\oplus b}\to M\to 0
\end{equation}
obtained from the minimal $R$-free resolution of $M$ in which $f$ is a matrix with entries in $\m$ and $i$ is the natural inclusion. Note that all the $R$-modules in~\eqref{eq20240912b} are non-zero since $\depth(R)=0$ and $M$ is not free, so $M$ must have infinite projective dimension. Applying the functor $\soc(-)=\Hom_R(k, -)$ to~\eqref{eq20240912b} and noting that $\soc(f)=0$, we get the equality $\soc(\syz^2(M))=\soc(R)^{\oplus a}$. Therefore, by~\eqref{eq20240912c} and~\eqref{eq20240912b} we must have
$$
\soc(R)^{\oplus a}=\soc(\syz^2(M))\subseteq \m \syz^2(M)\subseteq (\m^2)^{\oplus a}.
$$
Hence, $\soc(R) \subseteq \m^2$, which contradicts our assumption in~\eqref{int123}. Therefore, $k$ must be a direct summand of $\syz^2(M)$, as desired.
\qed

\begin{para}
From homological point of view, if the equivalent conditions~\eqref{int121}-\eqref{int123} in Theorem~\ref{int12} hold, then $R$ is Tor-friendly in the sense of~\cite{ains}, and hence, it is Ext-friendly; see~\cite[Propositions 3.3 and 5.5]{ains}.
\end{para}

As we mentioned in the introduction, if any of the equivalent conditions~\eqref{int121}-\eqref{int123} in Theorem~\ref{int12} holds, then by~\cite[Theorem 4.1]{bur} the ring $R$ is Burch with $\depth(R)=0$. The following result proves the same conclusion when $\syz^4$ is involved.

\begin{thm}\label{i}
If $k$ is a direct summand of $\syz^4(M)$ for every non-free $R$-module $M$, then $R$ is a Burch ring with $\depth(R)=0$. 
\end{thm}

\begin{proof}
Since $R$ is not a field, $k$ is not a free $R$-module.
It follows from the assumption that $k$ is a direct summand of $\syz^4(k)$, which implies that $\depth(R) = 0$.
Thus, by~\ref{1i} there is a positive integer $r$ such that $k^{\oplus r} \cong \syz^2(M)$ for a non-free $R$-module $M$. Hence, $\syz^2(k)^{\oplus r} \cong \syz^4(M)$.
By our assumption, $k$ is a direct summand of $\syz^2(k)^{\oplus r}$. Moreover, since $\End_R(k) \cong k$ is a local ring, using~\cite[Lemma 1.2]{lw} we conclude that $k$ is a direct summand of $\syz^2(k)$. Hence, the assertion that $R$ is a Burch ring follows from~\cite[Theorem 4.1]{bur}.
\end{proof}

\begin{para}
The converse of Theorem~\ref{i} is false by \cite[Example 4.6]{eda}. Also, we do not know if the condition on fourth syzygies in Theorem~\ref{i} can be shown to be equivalent to conditions~\eqref{int121}-\eqref{int123} of Theorem~\ref{int12}.
\end{para}

The proof of Theorem~\ref{i1} is split into Propositions~\ref{4}, \ref{thm20240912a}, and~\ref{prop20240915a} below. We proceed with the following discussion.

\begin{para}\label{para20240913b}
Dao and Eisenbud~\cite[Proposition 2.5]{eda} showed that the inequality
\begin{equation}\label{eq20240913a}
\B(R)\leq \edim(R)-\depth(R)
\end{equation}
holds for every local ring $R$ with equality holding when $R$ is Cohen-Macaulay with minimal multiplicty, among other conditions.
\end{para}

The next result, i.e., Theorem~\ref{i1} (a), introduces another case for which the equality holds in~\eqref{eq20240913a}, noting that here $\depth(R)=0$.

\begin{prop}\label{4}
If $\soc(R) \nsubseteq \m^2$, then  $\B(R)=\edim(R)$. 
\end{prop}

\begin{proof}
Without loss of generality, we can assume that $R$ is complete. Let  $R \cong S/I$ be a minimal Cohen presentation, that is, $(S, \n)$ is a regular local ring and $I \subseteq \n^2$ is an ideal of $S$. Specifying the ring in the colon ideals to avoid confusion, note that $\soc(S/I)=(I:_S \n)/I$ and the condition $\soc(R) \nsubseteq \m^2$ translates to $(I:_S \n) \nsubseteq \n^2$. Let $f \in (I:_S \n) \setminus \n^2$. Then, we have the containments
\begin{equation}\label{eq20240912e}
\left(I \n:_S (I:_S \n)\right) \subseteq (I \n:_S f) \subseteq (\n^3 :_S f).
\end{equation}
If $r \in (\n^3 :_S f)$, then $rf \in \n^3$, that is, $\ord_{\n}(rf) \geq 3$. (Here, $\ord_{\n}(x)$ is defined to be $\sup\{m\in \mathbb{Z}\mid x\in \n^m\}$; see~\cite[Definition 6.7.7]{sh}.) By~\cite[Theorem 6.7.8]{sh} we get
$$
\ord_{\n}(r) + \ord_{\n}(f) \geq 3.
$$
Since $f \notin \n^2$, we have $\ord_{\n}(f) \leq 1$. Therefore, $\ord_{\n}(r) \geq 2$, which implies that $r \in \n^2$, i.e., $(\n^3 :_S f)\subseteq \n^2$. It follows from~\eqref{eq20240912e} that $\left(I \n:_S (I:_S \n)\right) \subseteq \n^2$. The other inclusion $\n^2 \subseteq \left(I \n:_S (I:_S \n)\right)$ is immediate since $\n^2(I:_S\n)\subseteq \n I$, and hence, we have the equality $\n^2 = \left(I \n:_S (I:_S \n)\right)$.
Thus, we have a series of equalities
$$
\B(R)= \dim_k \left(\n/\left(I \n:_S (I:_S \n)\right)\right) = \dim_k \left(\n/\n^2\right) = \edim(R)
$$
in which the first equality comes from the definition of Burch index.
\end{proof}


Part (b) of Theorem~\ref{i1} can be deduced from the next proposition. This result clearly shows that the equivalent conditions~\eqref{int121}-\eqref{int123} in Theorem~\ref{int12} are, in some sense, restrictive when the ring is assumed to be Gorenstein.

\begin{prop}\label{thm20240912a}
Assume that $R$ is Gorenstein. If $k$ is a direct summand of $\syz^2(M)$ for every non-free $R$-module $M$, then $R$ is an artinian hypersurface of the form $R\cong S/(f^2)$, where $S$ is a discrete valuation domain and $0\neq f\in S$. In particular, $R$ has minimal multiplicity.
\end{prop}

\begin{proof}
	Combining the fact that $R$ is Burch in this case with~\cite[Proposition 5.1]{bur}, we deduce that $R$ is a hypersurface. Note that $\dim(R) = \depth(R) = 0$, that is, $R$ is artinian. Therefore, there exist a regular local ring $(S,\fn)$ with $\dim(S)=1$ (i.e., $S$ is a discrete valuation domain) and a principal ideal $I$ of $S$ such that $R\cong S/I$. Let $I = (f^n)$, where $f \in \fn$ is the uniformizer and $n$ is a positive integer. Note that $n\neq 0, 1$ because the artinian ring $R$ is not the zero ring and is not regular (i.e., is not a field). If $n\geq 2$, then using the zero-divisor pair $f$ and $f^{n-1}$ on $R$, we can construct an $R$-module $M$ such that $\syz^2(M) = fR$. More precisely, considering the exact sequence
$$
0\to fR\xra{\alpha} R\xra{\beta}R\to \coker(\beta)\to 0
$$
of $R$-modules in which $\alpha$ is the natural injection and $\beta$ is the multiplication map by $f^{n-1}$, we have $M=\coker(\beta)$.
By assumption, $fR$ admits a $k$-summand, but as $fR$ is principal, this implies that $fR\cong k$, so $f^2\in I$, forcing $n=2$.
\end{proof}

Before we proceed to the proof of Theorem~\ref{i1} (c), i.e., Proposition~\ref{prop20240915a} below, we remind the reader of the definition of reflexive modules.

\begin{para}\label{para20240915v}
	An $R$-module $M$ is called \emph{reflexive} if the biduality map $f\colon M\to M^{**}$ defined by $f(m)\varphi=\varphi(m)$, for all $\varphi\in M^*$, is an isomorphism.
\end{para}

We also need the following tool for the proof of Theorem~\ref{i1} (c).

\begin{para}\label{para20240911c}
Let $F_1\xra{\varpi} F_0\to M\to 0$ be a free presentation for an $R$-module $M$. Applying the functor $(-)^*$, we get the exact sequence
$$
0\to M^*\to F_0^*\xra{\varpi^*} F_1^*\to \coker(\varpi^*)\to 0.
$$
We denote $\coker(\varpi^*)$ by $\Tr M$; it is called the \emph{Auslander transpose of $M$}.
The Auslander transpose depends on the free presentation of $M$, however, it is unique up to free summand.
Also, the $R$-modules $\Tr \Tr M$ and $M$ are isomorphic up to free summands.
\end{para}

\begin{prop}\label{prop20240915a}
Suppose $k$ is a direct summand of $\syz^2(M)$ for all non-free $R$-modules $M$. If $R$ is Cohen-Macaulay and non-Gorenstein and $M$ is a reflexive $R$-module, then $M$ is free.
\end{prop}

\begin{proof}
Suppose, on the contrary, that there exists a non-free reflexive $R$-module $M$. If $\Tr M^*$ is free, then $\Tr \Tr M^*$ is also free and by our discussion in~\ref{para20240911c}, the $R$-module $M^*$ is free as well. Consequently, $M\cong M^{\ast\ast}$ is free, which contradicts our assumption that $M$ is non-free. Therefore, $\Tr M^*$ is not a free $R$-module.
Thus, $k$ is a direct summand of $\syz^2(\Tr M^\ast) = M^{\ast\ast} \cong M$. Since any direct summand of a reflexive module is also reflexive, $k$ must be reflexive. This implies that $k^\ast \neq (0)$, and hence, $\depth(R) = 0=\dim(R)$. Given that $k$ is reflexive, we deduce that $\type(R) = 1$, which means that $R$ is Gorenstein. This contradiction shows that such an $R$-module $M$ does not exist.
\end{proof}

The assumption that $R$ is Cohen-Macaulay is necessary in Proposition~\ref{prop20240915a}; this is demonstrated by Example~\ref{PropCMNecessity}.

\section{Examples and more}\label{sec20240912z}

This section is entirely devoted to listing a variety of examples that do and do not satisfy conditions~\eqref{int121}-\eqref{int123} in Theorem~\ref{int12} as well as some examples that deal with higher syzygy modules. We refer the reader to~\cite{bur} for more examples. We also pose Question~\ref{question20240913a} that may be of interest to a curious reader. We start with the following example that introduces a large class of local rings satisfying conditions~\eqref{int121}-\eqref{int123} in Theorem~\ref{int12}; compare with~\cite[Theorem 4.1(1)]{eda}.

\begin{ex}\label{3}
Let $R=S\ltimes k$ be the Nagata idealization (also called trivial extension) of a local ring $(S,\n, k)$ with respect to its residue field $k$. As an underlying set, $R$ is the cartesian product $S\times k$, with multiplication defined by the formula $(s_1,\lambda_1)(s_2,\lambda_2)=(s_1s_2,s_2\lambda_1+s_1\lambda_2)$. Note that $\m= \n \oplus k$ and the elements of $\m^2$ always have second coordinate equal to $0$. Thus, $(0,1)\in \soc(R)\setminus \m^2$, which means that $\soc(R)\nsubseteq \m^2$. 
\end{ex}

\begin{para}\label{para20240914a}
Using the notation from Example~\ref{3}, note that there is a ring isomorphism
$$
R\cong S\times_k\left(k[x]/(x^2)\right)
$$
in which the right-hand side is the fiber product ring over $k$; see, for instance, \cite{nasseh:vetfp}-\cite{nasseh:oeire} for the definition and various properties of the fiber product rings.
\end{para}

Having the fiber product rings introduced, the next example shows that the Cohen-Macaulay assumption is necessary in Proposition~\ref{prop20240915a}, i.e., in Theorem~\ref{i1} (c).
More generally, any ring with $\soc(R)=(x)\not\subseteq\m^2$ that is not Cohen-Macaulay demonstrates the necessity of the Cohen-Macaulay hypothesis in Proposition~\ref{prop20240915a}.

\begin{ex}\label{PropCMNecessity}
Let $R=k[\![x,y]\!]/(xy,y^2)$, which is not Cohen-Macaulay. It is straightforward to check that $R\cong k[\![x]\!]\times_k \left(k[\![y]\!]/(y^2)\right)$, which is a fiber product ring over $k$; for this isomorphism see, for instance, \cite[3.3]{NST}. As we discussed in~\ref{para20240914a}, we have an isomorphism $R\cong k[\![x]\!]\ltimes k$ and therefore, by Example~\ref{3} of the next section, $k$ is a direct summand of $\syz^2(M)$ for all non-free $R$-modules $M$. Note that $k^*\cong \soc(R)=Rx\cong R/\ann(x)=k$. Thus, $k^{\ast\ast}\cong k$, and so $k$ is a reflexive $R$-module which is not free. 

Note that over this particular ring $R$ of depth zero, none of the $R$-modules $\syz^n(k)$ with $n\geq 1$ is reflexive in view of ~\cite[Theorem 4.1(3)]{deyt}.  
\end{ex}

\begin{ex}
If $R=S/ f\n$, where $(S,\n)$ is a local ring and $f\in\n^t\setminus \n^{t+1}$ for some positive integer $t$, then $k$ is a direct summand of $\m^t$. This is another application of Lemma~\ref{HVS} because $f\in\soc(\m^t)\setminus \m^{t+1}$. Hence, in the special case where $t=1$, we have that $k$ is a direct summand of $\m$, i.e., conditions~\eqref{int121}-\eqref{int123} in Theorem~\ref{int12} hold.
\end{ex}

The following example is specifically interesting in that it shows how strong $k$ being a direct summand of $\m=\syz^1{k}$, namely condition~\eqref{int124} in Theorem~\ref{int12}, is. As we see in this example, the residue field $k$ being a direct summand of higher syzygies of $k$ is not enough for conditions~\eqref{int121}-\eqref{int123} in Theorem~\ref{int12} to hold.\vspace{2mm}

\begin{ex}\label{ex20240913b}
Let $R=k[\![x,y]\!]/(x^4,x^2y,y^2)$. Using Macaulay2\cite{Grayson}, after performing some elementary row and column operations, we compute that a presentation matrix of $\syz^2(k)$ is
$$
\begin{pmatrix}
y&x&0&0&0&0&0&0\\
0&-y&xy&x^3&0&0&0&0\\
0&0&0&0&y&x&0&0\\
0&0&0&0&0&0&y&x
\end{pmatrix}
=\begin{pmatrix}
y&x&0&0\\
0&-y&xy&x^3
\end{pmatrix}
\oplus\begin{pmatrix}
y&x
\end{pmatrix}^{\oplus2}.
$$
Note that the matrix 
$\begin{pmatrix}
y&x&0&0\\
0&-y&xy&x^3
\end{pmatrix}$
is the presentation matrix of $\m$. Therefore, we have the isomorphism
\begin{equation}\label{eq20240913d}
\syz^2(k)\cong\m\oplus k^{\oplus2}=\syz^1(k)\oplus k^{\oplus2}.
\end{equation}
Applying the functor $\syz^1(-)$ repeatedly to~\eqref{eq20240913d}, we conclude that $k$ is a direct summand of $\syz^n(k)$ for all integers $n\geq 2$.
On the other hand, $\soc(R)=(xy,x^3)\subseteq\fm^2$. Therefore, it follows from Lemma~\ref{HVS} that $k$ is not a direct summand of $\fm=\syz^1(k)$. This means that conditions~\eqref{int121}-\eqref{int123} in Theorem~\ref{int12} do not hold for this example. This fact confirms the conclusion of~\cite[Example 4.5]{eda}.
\end{ex}

The following example shows that a condition on first syzygies can not be added to the equivalent conditions of Theorem~\ref{int12}.

\begin{ex}\label{ex20240913a}
Let $M$ be a non-free $R$-module, and for each non-negative integer $n$ let $\beta_{n}(M)$ denote the $n$-th Betti number of $M$, that is, $\beta_{n}(M)=\dim_k(\Tor^R_n(k,M))$. If $\fm^2=0$, then $k^{\beta_{n+1}(M)}$ is a direct summand of $\syz^n(M)$. To see this, note that $\syz^n(M)\subseteq \fm R^{\beta_{n}(M)}$. Therefor, $\syz^n(M)$ is a $k$-vector space with $\dim_k(\syz^n(M))=\beta_{n+1}(M)$. Hence, $k^{\beta_{n+1}(M)}\cong\syz^n(M)$, and in particular, $k$ is a direct summand of $\syz^1(M)$ for every non-free $M$.

Conversely, if $k$ is a direct summand of $\syz^1(M)$ for every non-free $R$-module $M$, then $k$ is a direct summand of $(x)=\syz^1(R/(x))$ for every $x\in \m$.
But since $(x)$ is principle, this means $(x)\cong k$ is killed by $\m$, and since $x\in\m$ was arbitrary, this means $\m^2=0$.
\end{ex}

If $R$ is non-Gorenstein with $\edim(R)=2$, then by~\cite[Example 1.2]{link} the $R$-module $k^{\oplus \type(R)}$ is a direct summand of $\syz^3(k)$.
This fact and Example~\ref{ex20240913a} motivate us to ask the following question for which we do not have a general answer at this time, even for $M=k$, although \cite[Theorem~B]{DBM} provides a course lower bound for modules over Golod rings.

\begin{question}\label{question20240913a}
Assume that $R$ is Burch with $\depth(R)=0$, and let $M$ be a non-free $R$-module (e.g., $M=k$).
How many copies of $k$ occur as direct summands in $\syz^n(M)$ for each positive integer $n$?
\end{question}

Relevant to Question~\ref{question20240913a}, our computation in Example~\ref{ex20240913b} implies the following.

\begin{ex}
Consider the ring $R$ from Example~\ref{ex20240913b}. Define a sequence $\{a_i\}_{i\in \mathbb{N}}$ of positive integers recursively as follows: Let $a_1=0$, $a_2=2$, and for all integers $n\geq 1$ define $a_{2n+1}=4a_{2n-1}+2$ and $a_{2n+2}=4a_{2n}-2$. Then, for all integers $n\geq 1$, the $R$-module $k^{a_n}$ is a direct summand of $\syz^n(k)$. Moreover, $a_n$ is maximal with respect to this property, that is, $a_n$ copies of $k$ occur as direct summands in $\syz^n(k)$.
\end{ex}

According to~\ref{para20240914a}, the next example is a special case of Example~\ref{3} with $S=k[\![y]\!]/(y^n)$. However, since it contains more information relevant to Question~\ref{question20240913a}, we describe it separately here.

\begin{ex}
For a positive integer $n\geq 3$ let $R=k[\![x,y]\!]/(x^2,xy,y^n)$ and note that $R\cong k[\![x]\!]/(x^2)\times_k k[\![y]\!]/(y^n)$; see, for instance, \cite[3.3]{NST}. Then, as we mentioned in Example~\ref{3}, the residue field $k$ is a direct summand of $\syz^2(M)$ for all non-free $R$-modules $M$.
Also, note that $\soc(R)\nsubseteq \fm^2$. Thus, by Lemma~\ref{HVS}, the residue field $k$ is a direct summand of $\fm=\syz^1(k)$ as well. Since $\fm^2\neq 0$, we deduce that $k^{\oplus 2}$ is not a direct summand of $\fm$.
In view of~\cite[Proposition 4.1] {link} we have
$$
\syz^2(k)=\syz^1\fm\cong\m\oplus k^{\oplus2}.
$$
Thus, $k^{\oplus3}$ is a direct summand of $\syz^2(k)$ while $k^{\oplus4}$ is not a direct summand of $\syz^2(k)$. From this, and by an inductive argument for all $n\geq 3$ we have
$$
\syz^n(k)\cong \syz^{n-2}(k^{\oplus2})\oplus \syz^{n-1}(k).
$$
We conclude that, for all integers $n\geq 3$, the number of copies of $k$ appearing as direct summands of $\syz^n(k)$ is given by the sequence $b_n=2b_{n-2}+b_{n-1}$ with $b_1=1$ and $b_2=3$.
\end{ex}


\begin{ex}\label{213}
Let $R= S/I\fn$, where $S=k[\![x]\!]$, $\n=(x)$, and $I= (x^3)$ is an ideal of $S$. Using the zero-divisor pair $a= x^2=b$, we can construct an $R$-module $M$ such that $\syz^6(M)= aR$. Note that $\syz^6(M)$ does not admit $k$ as a submodule. In particular, $k$ is not a direct summand of $\syz^6(M)$. Note that by~\cite[Position 2.5(2)]{eda} we have $\B(R)=1$. Hence, this does not contradict~\cite[Theorem A]{DBM}.
\end{ex}

 


\begin{ex}\label{daoeis}
Let $R=\mathbb{F}_2[\![x,y,z]\!]/(x,y^2,z^3)^2$ and consider the $R$-module $M=(x,y^2,z^3)R$. Using Macaulay2 and performing some elementary row and column operations, we find that the presentation matrix for $M$ is given by $\begin{pmatrix} x & y^2 & z^3 \end{pmatrix}^{\oplus 3}$. Thus, $\syz^1(M)\cong M^{\oplus 3}$. Consequently, $\syz^n(M)\cong M^{\oplus 3^n}$ for all integers $n\geq 1$. On the other hand, it follows from Lemma~\ref{HVS} that $k$ is not a direct summand of $M$. Therefore, $k$ is not a direct summand of $\syz^n(M)$ for all integers $n\geq 1$ either.
\end{ex}


\section{The $R$-module $E^*$ being annihilated by maximal ideal}\label{sec20240911a}

The rest of this paper is devoted to our work on Question~\ref{question20240911a} which is motivated partly by our discussion in~\ref{cor20230918a} and partly by the fact that $E^*\cong \syz^2\left(\Tr E\right)$, where the latter is directly related to our work in the previous sections.
The main results of this section are summarized in Proposition~\ref{g} and Theorem~\ref{42new}.

\begin{para}\label{cor20230918a}
Assume that $R$ is artinian with $\ell\ell(R)=n+1$, for some integer $n\geq 1$. Then, $R$ is Gorenstein if and only if $\fm^{n}E^*\neq 0$. In fact, if $R$ is Gorenstein, then $E \cong R$ and hence, $\fm^n E^* \cong \fm^n \neq 0$. For the converse, if $\fm^n E^* \neq 0$, then there exists $\alpha \in E^*$ such that $\fm^n \im \alpha \neq 0$. The condition $\fm^{n+1} = 0$ forces $\alpha$ to be surjective, and therefore $R$ is a direct summand of $E$. Hence, $R$ is Gorenstein.
\end{para}

If $R$ is artinian and $E^*$ is a $k$-vector space, then its vector space dimension can be computed as we show in the next proposition. This result is well-known to experts; see for instance, \cite[Proof of Proposition 3.14]{Lyle}. However, we give the proof for the convenience of the reader.

\begin{prop}\label{siz}
Assume that $R$ is artinian, and let $M$ be an $R$-module with $\m M^{*}=0$. Then, $\dim_k(M^*)=\lambda_R(M)\type(R)$. In particular, if $E^*$ is a $k$-vector space, then $\dim_k(E^*)=\left(\type(R)\right)^2$. 
\end{prop}

\begin{proof}
Let $n=\dim_k(M^*)$. By the Hom-tensor adjointness and the fact that $E^\vee \cong R$, we have $M^* \cong \left(M \otimes_R E\right)^\vee$. Hence, $k^{\oplus n}\cong \left(M \otimes_R E\right)^\vee$. It follows from taking another Matlis dual that $k^{\oplus n}\cong M \otimes_R E$. Now, we have the equalities
\begin{align*}
n=\lambda_R\left(k^{\oplus n}\right)&=\lambda_R(M \otimes_R E)\\
&=\dim_k(M \otimes_R E\otimes_Rk)\\
&=\dim_k\left((M\otimes_Rk)\otimes_k(E\otimes_Rk)\right)\\
&=\dim_k(M\otimes_Rk)\dim_k(E\otimes_Rk)\\
&=\lambda_R(M)\lambda_R(E)\\
&=\lambda_R(M)\type(R)
\end{align*}
which complete the proof of the proposition. 
\end{proof}

 Next result introduces a large class of artinian local rings satisfying the equality $\m E^*=0$ in Question~\ref{question20240911a}. After the first draft of this paper appeared on arXiv, we were informed by Toshinori Kobayashi that a part of this result was obtained by Ananthnarayan~\cite[Theorem 3.2]{anan} in his computation of the Gorenstein colength. However, we provide different proofs, one of which uses the notion of Eliahou-Kervaire resolution, and the other is a combinatorial proof which uses Macaulay inverse systems and may be of independent interest.

\begin{prop}\label{g}
	Let $R=S/\n^n$ for an integer $n\geq 2$, where $S = k[x_1,\dots,x_e]$ with $e\geq 2$ and $\n=(x_1,\dots,x_e)$. Then, $\m E^*=0$ with
	\begin{equation}\label{eq20240917a}
	\dim_k(E^*)=\binom{e+n-2}{n-1}^2.
	\end{equation}
\end{prop}

Before proving  Proposition~\ref{g}, it is beneficial to see an example.

\begin{ex}\label{ex41}
Following the notation from  Proposition~\ref{g}, let $S=k[x,y]$ and $n=3$, so that $R=k[x,y]/(x,y)^3$. Note that $E\cong \Ext^2_S(R,S)\cong \Ext^1_S(\n^3,S)$. Since $\depth_S(\n^3)=1$, by the Auslander-Buchsbaum formula, $\pd_S(\n^3)=1$. In fact, the minimal free resolution of $\n^3$ is
\begin{equation}\label{eq20240920a}
0\to S^3\xra{A} S^4\to \n^3\lo 0
\end{equation}
in which the map $A$ is a matrix given by
\begin{equation*}
A=\begin{pmatrix}
-y	& 0   &  0  \\
x 	&	-y    &    0\\
0 	&	x    &    -y\\
0 	&	0    &    x
\end{pmatrix}.
\end{equation*}
Applying the functor $\Hom_S(-,S)$ to~\eqref{eq20240920a}, we get an exact sequence
\begin{equation}\label{eq20240920b}
S^{4}\xra{A^{\tr}} S^{3}\to E\to 0
\end{equation}
of $S$-modules. Since $\n^3 E=0$, applying $-\otimes_SR$ to~\eqref{eq20240920b} we obtain a  
presentation 
\begin{equation}\label{eq20240920c}
R^{4}\xra{\overline{A^{\tr}}} R^{3}\to E\to 0
\end{equation}
of $E$. Note that $A$ is of linear type and its entries are not in $\n^3$.
In particular, no column or row of $\overline{A^{\tr}}$ is zero. Thus,
the presentation~\eqref{eq20240920c} is minimal. Applying $(-)^*=\Hom_R(-,R)$ to~\eqref{eq20240920c}, we get an exact sequence
$$
0\to E^*\to R^{3} \xra{\overline{A}} R^4
$$
of $R$-modules, i.e., $E^*=\ker(\overline{A})$.
Hence, $E^*$ is generated by $9$ elements of $R^3$ that form the columns of the matrix
\begin{equation*}
B=\begin{pmatrix}
y^2   	& xy   &  x^2   & 0	  & 0   &  0 & 0   & 0   &  0    \\
0	&	0    &    0&y^2   	& xy   &  x^2  & 0   	& 0   & 0   \\
0 	&	0    &    0&0  	& 0   &  0	& y^2   	& xy   &  x^2  
\end{pmatrix}.
\end{equation*} 
Since $\fm^3=0$ and entries of $B$ are in $\fm^2$, they get annihilated by $\fm$. Thus, $\fm E^* = 0$.
\end{ex}

 Our first proof of Proposition~\ref{g}, uses the Eliahou-Kervaire resolution, which is the minimal resolution of a large class of monomial ideals; see~\cite{EK, MR1407879, peeva} for the definition and properties of the Eliahou-Kervaire resolution. We adopt the notation and conventions of \cite{peeva}, but summarize a few important details and specialize to ideals of the form $(x_1,\dots,x_e)^n$.

\begin{para}\label{para20240916d}
We work in the setting of  Proposition~\ref{g}.
For a monomial $f$, we let $\min(f)$ and $\max(f)$ be the minimum and maximum indices, respectively, of variables dividing $f$.
The basis elements of the minimal free resolution of $\n^n$ may be enumerated by labels $(f;\,j_1,\dots,\,j_i)$, where $f$ is a monomial of degree $n$ and we have the inequalities $1\leq j_1<\dots<j_i<\max(f)$.
The homological degree of such a basis element is given by the sum $j_1+\cdots+j_i$.

Any monomial $f$ of degree at least $n$ can be uniquely factored as $b(f)\cdot g$ with $\max(b(f))\leq \min(g)$ and $\deg(b(f))=n$.
The differential of the Eliahou-Kervaire resolution is given by $\partial=d-\delta$, where $d$ and $\delta$ are defined as follows:
\begin{gather}
\delta(f;j_1,\dots,j_i)=\sum\limits_{\ell=1}^i(-1)^\ell\frac{fx_\ell}{b(fx_\ell)}(b(fx_\ell);j_1,\dots,\hat{j_\ell},j_i)\label{eq20240920d}\\
d(f;j_1,\dots,j_i)=\sum\limits_{\ell=1}^i(-1)^\ell x_i(f;j_1,\dots,\hat{j_\ell},j_i). 
\end{gather}
\end{para}

\noindent \emph{First proof of Proposition~\ref{g}.} We consider the setting of~\ref{para20240916d}. Let $F\xra{\simeq} R$ be the Eliahou-Kervaire resolution of $R$ over $S$ with the last map described by the matrix $A=\partial_e$. For degree reasons, the columns of $A$ are indexed by $(gx_e;1,\dots,e-1)$ where $g$ is any monomial of degree $n-1$. 

For each monomial $g$ of degree $n-1$ and each $x_\ell$ we have the equality $b(gx_\ell x_e)=gx_\ell$. Hence, the formula~\eqref{eq20240920d} for $\delta_e$ simplifies to
$$
\delta_e(gx_e;1,\ldots,e-1)=\sum_{i=1}^{e-1} x_e(gx_i;1,\dots,\hat{i},\dots{e-1}).
$$
Thus, every entry of $A$ is linear. As in Example~\ref{ex41}, it follows that $\overline{A^{\tr}}$ is a minimal presentation matrix for $E$, which implies that $E^*=\ker\left(\overline{A}\right)$, where $\overline{A}=A\otimes_S R$.\vspace{2mm}

\noindent {\bf Claim}: We have the equality $\ker(\overline{A})=(0:_R\m)\overline{F}_e$.\vspace{2mm}

Once we prove the claim, the assertion of  Proposition~\ref{g} follows readily.

Note that the containment $\supseteq$ holds since the entries of $\overline{A}$ are in $\m$. For the reverse containment, we show that a maximal square submatrix of $A$ is triangular with linear forms down the diagonal. Degree reasons then force that if $\overline{A}(b)=0$, then each entry of $\overline{A}(b)$ has degree $n$, whence each entry of $b$ has degree $n-1$. Consequently, $b\in (0:_R\m)\overline{F}_e$, as desired.

To obtain the desired triangular submatrix of $A$, we order the basis elements of the Eliahou-Kervaire resolution lexicographically within each homological degree, giving preference to the monomials over the index sets. Namely, to compare basis elements $(f;j_1,\dots,j_i)$ and $(g;\ell_1,\dots,\ell_i)$, we compare $f$ and $g$ lexicographically, and if $f=g$, we then compare $(j_1,\dots,j_i)$ with $(\ell_1,\dots,\ell_i)$. With this ordering, the smallest basis element appearing with a nonzero coefficient in $\del(gx_e;1,\dots,e-1)$ is $(gx_e;1,\dots,e-2)$. Whenever the column $(g_1x_e;1,\dots, e-1)$ is less than the column $(g_2x_e;1,\dots, e-1)$, it is evident that $(g_1x_e;1,\dots,e-2)$ is less than $(g_2x_e;1,\dots,e-2)$. Selecting the rows $(gx_e;2,\dots,e-1))$, where $g$ ranges over all monomials of degree $n-1$, ordered so that the lowest monomials appear highest in the matrix, then gives a lower triangular maximal submatrix of $A$, as desired.

The equality~\eqref{eq20240917a} follows from Proposition~\ref{siz} and the fact that in this case $\type(R)=\binom{e+n-2}{n-1}$.
\qed

\begin{ex}
Following the notation from~\ref{para20240916d} and the proof of  Proposition~\ref{g}, for $e=4$ and $n=4$ we have
\begin{gather*}
d((x_1x_2^2x_4;1,2,3))=-x_1(x_1x_2^2x_4;2,3)+x_2(x_1x_2^2x_4;1,3)-x_3(x_1x_2^2x_4;1,2)\\
\delta((x_1x_2^2x_4;1,2,3))=x_4\left(-(x_1^2x_2^2;2,3)+(x_1x_2^3;1,3)-(x_1x_2^2x_3;1,2)\right)
\end{gather*}
and we see that the smallest basis element appearing with a non-zero coefficient in $\del((x_1x_2^2x_4;1,2,3))$ is $(x_1x_2^2x_4;1,2)$.
\end{ex}

In addition to satisfying the equality $\m E^*=0$, the ring $R$ in Proposition~\ref{g} is Burch. Here is an example of an artinian
Burch ring $R$ with $\m E^*\neq 0$, which at the same time shows that the assumption $e\geq 2$ in  Proposition~\ref{g} is necessary.

\begin{ex}\label{ex20240917d}
Let $R=k[\![x]\!]/(x^2)$. Then, $R$ is Burch and $\m E^*\cong \m\neq 0$.
\end{ex}

 Our second proof of Proposition~\ref{g} involves relating the combinatorial structure of strongly stable ideals with that of Macaulay inverse systems; see~\cite{MacaulayModular} or~\cite[\textsection~21.2]{CommutativeAlgebra}. Furthermore, this proof may better generalize to other Borel fixed ideals.

\begin{para}\label{para20240920a}
Consider an integer $e\geq 2$. For integers $i,j$ with $1\leq i<j\leq e$, multiplication of a monomial $m$ in $S=k[x_1,\dots,x_e]$ by $x_ix_j^{-1}$ with $x_j$ dividing $m$ defines a \emph{Borel Move}. The \emph{inverse Borel move} is defined to be  $x_jx_i^{-1}$. Moreover, \emph{Borel-fixed ideals} are monomial ideals of the polynomial ring $S$ which are closed under Borel moves.
\end{para}

\begin{para}\label{para20240920b}
Consider the notation from~\ref{para20240920a}, and let $I$ be a monomial ideal of $S$ such that $R=S/I$ is artiniain. Borel moves extend to the inverse system of $I$ as follows:

Let $S$ act on its fraction field $k(x_1,\dots,x_e)$ in the natural way, and view $\omega_S=k[x_1^{-1},\dots,x_e^{-1}]$ as the quotient by $(x_1,\dots,x_e)$ viewed as a submodule of $k(x_1,\dots,x_e)$. For integers $1\leq i<j\leq e$, the Borel move on an inverse monomial $m\in \omega_S$ is computed by multiplying $x_ix_j^{-1}m$ in $k(x_1,\dots,x_e)$ and reducing modulo $S$.
\end{para}

\begin{para}\label{p:InverseBorel}
Consider the notation from~\ref{para20240920b}. Inverse Borel moves in $S$ correspond to Borel moves on $\omega_S$, and this descends to $R$. In fact, Borel moves on ideals of $R$ correspond to inverse Borel moves on $\omega_S$-submodules containing $(0\colon_{\omega_S} I)$, and vice-versa. Note that the latter is isomorphic to $E$ once tensored down to $R$.
\end{para}

The following lemma holds more generally for an arbitrary artinian ring: 

\begin{lem}\label{p:ExactSequence}
Consider an integer $e\geq 2$, and assume that $R$ is artinian. Then, there exists an exact sequence
\begin{equation}\label{eq20240920e}
0\to L\to E^*\to M\to 0
\end{equation}
in which $\m L=0$ with $\dim_k(L)=(\type(R))^2$.
\end{lem}

\begin{proof}
Since $E$ is generated by $\type(R)$-many elements, there exists a surjection $E\to k^{\oplus \type(R)}$, which dualizes to an injection $L=(k^*)^{\oplus \type(R)} \to E^*$. Note that $k^*\cong \soc(R)$ and $\dim_k(k^*)=\type(R)$. Therefore, $\dim_k(L)=(\type(R))^2$.
\end{proof}


\begin{para}\label{p:BorelGeneration}
Under the setting of  Proposition~\ref{g}, $\n^{n-1}$ is closed under Borel and inverse Borel moves, and hence, so is $E$. Applying Borel moves and inverse Borel moves to a minimal generator of $E$ yields another minimal generator of $E$ and moreover, every minimal generator of $E$ can be obtained from any other by a sequence of Borel and/or inverse Borel moves. These assertions of closure and preservation of minimal generators follow from the fact that Borel moves preserve degrees.
\end{para}

\begin{lem}\label{lem20240920a}
Consider the setting of  Proposition~\ref{g}. The module $M$ introduced in Lemma~\ref{p:ExactSequence} is zero.
\end{lem}

\begin{proof}
Note that $\soc(R)=\n^{n-1}$, and so $E$ is generated by all inverse monomials of degree $-n+1$. Given two generators $m$ and $m'$ of $E$ which are related by an (inverse) Borel move $m'=x_jx_i^{-1}m$, for any element $\phi\in E^*$ we have $x_i\phi(m)=x_j\phi(m')$.
If $m$ is killed by $x_i$, this shows that $x_j$ kills $\phi(m')$. However, for the ideal $\n^n$, Borel moves and their inverses act transitively on the socle generators.
Since each of the corresponding inverse monomials $x_j^{-(n-1)}$ are killed by $x_i$ for all $i\neq j$, this shows that provided $e\geq 2$, the element $\phi(m)$ is killed by $x_i$ for all $i$ and all socle generators $m$, as desired. 
\end{proof}

\noindent \emph{Second proof of  Proposition~\ref{g}.} This follows readily from Lemmas~\ref{p:ExactSequence} and~\ref{lem20240920a}.\qed

We present an example which uses the notation $\tr_R(-)$ from~\ref{para20240911b}.
\begin{ex}
Let $R=S/(f_1,\ldots,f_d)^n$, where $S=k[x_1,\ldots,x_d]$ with $d>1$ and $f_1,\ldots,f_d$ is a system of parameters. Then, $R$ is a Golod ring and $\frak \fm E^\ast=0$ if and only if $\deg(f_i)=1$. In fact, it follows from~\cite[Theorem 3.13]{anan} that
$\tr_R(E) = (f_1,\ldots,f_d)^{n-1}/(f_1,\ldots,f_d)^n$. This implies that $\frak m E^\ast=0$ if and only if $\frak m (f_1,\ldots,f_d)^{n-1}\subseteq (f_1,\ldots,f_d)^n$, which gives the assertion.
\end{ex}

The rest of this section is devoted to prove Theorem~\ref{42new}. To this end, we start with the following general lemma that will be used later in this section.


\begin{lem}\label{lem20250129a}
Let $M$ be a finite length $R$-module such that $\soc(M)$ is 1-dimensional. If there exist an $R$-module $N$ and morphisms $\eta\colon k \to M$ and $\alpha\colon M \to N$ such that $\alpha \eta\neq 0$, then $\alpha$ is injective.
\end{lem}

\begin{proof}
Note that $\eta \neq 0$ and hence, $\eta(1)\neq 0$. Moreover, since $\eta(1)\in \soc(M)$ and $\soc(M)$ is 1-dimensional, we have $\soc(M)=R\eta(1)=\im \eta$. Let $0\ne x \in M$. Since $\soc(M)$ is an essential submodule of $M$, we have $Rx \cap R \eta(1)\neq 0$, i.e., there exist elements $r_x,s_x\in R$ such that $r_xx=s_x\eta(1)\neq 0$. Since $\eta(\m)=0$, we have $s_x\notin \m$ and thus, $s_x^{-1}\in R$ with $\eta(1)=s_x^{-1}r_xx$. Therefore, $s_x^{-1}r_x\alpha(x)=\alpha(\eta(1))$ which is non-zero by our assumption. Hence, $\alpha(x)\neq 0$ and since $x\in M$ was arbitrary, this proves $\alpha$ is injective, as desired.
\end{proof}

\begin{para}\label{conv20250129a}
Let $R=S\times_k T$ be the fiber product of the artinian local rings $(S,\m_S,k)$ and $(T,\m_T,k)$ with the canonical modules $E_R$, $E_S=E_S(k)$, and $E_T=E_T(k)$; note that $R$ is also artinian with maximal ideal $\m_R=\m_S\oplus \m_T$ and we have the $R$-module isomorphisms $S\cong R/\m_T$ and $T\cong R/\m_S$. Consequently, any $S$-module $M$ is annihilated by $\m_T$ as an $R$-module, and similarly,  any $T$-module $N$ is annihilated by $\m_S$ as an $R$-module. In fact, $R$ fits in the pull-back diagram 
\begin{equation}\label{eq20250129a}
\begin{split}
\xymatrix{
R\ar[r]\ar[d]&T\ar[d]^{\pi_T}\\
S\ar[r]^{\pi_S}&k
}
\end{split}\end{equation}
where $S\xra{\pi_S} k\xla{\pi_T}T$ are the natural surjections.
By~\cite{og}, we know that $E$ fits into the push-out diagram
\begin{equation}\label{eq20250129b}
\begin{split}
\xymatrix{
k\ar[r]^{\eta_S}\ar[d]_{\eta_T}&E_S\ar[d]^{\iota_S}\\
E_T\ar[r]^{\iota_T}&E_R
}
\end{split}\end{equation}
arising from applying $(-)^{\vee}$ to~\eqref{eq20250129a}. We also note that if $R$ is not Gorenstein, i.e., $R$ is not a summand of $E_R$, then 
\begin{equation}\label{eq20250129d}
E_R^*=\Hom_R(E_R,\m_R)\cong \Hom_R(E_R,\m_S)\oplus \Hom_R(E_R, \m_T).
\end{equation}
\end{para}

\begin{lem}\label{lem20250129d}
Consider the setting of~\ref{conv20250129a}. If $S$ is not a field, then $\Hom_R(E_S,\m_T)$ is a direct summand of $\Hom_R(E_R,\m_T)$. In particular, if $R$ is not Gorenstein, then $E_R^*$ has a nonzero direct summand annihilated by $\m_R$ which has dimension $\type(S)\type(T)$ over $k$.
\end{lem}

\begin{proof} Throughout the proof, we will write $E$ for $E_R$. 
If $T$ is a field, then $\mathfrak{m}_T = 0$, and the statement is trivial. Hence, to proceed, we assume that $T$ is not a field. Let $\alpha \in \Hom_R(E_S, \mathfrak{m}_T)$ and note that $E_S$ is indecomposable and not simple because $S$ is not regular. Thus, $\soc(E_S) \subseteq \mathfrak{m}_S E_S$. Since $\mathfrak{m}_S\mathfrak{m}_T=(0)$, we conclude that $\alpha \eta_S = 0$. In fact,
$$
\im(\alpha \eta_S) = \alpha(\soc(S)) \subseteq \alpha(\mathfrak{m}_S E_S) = \mathfrak{m}_S \alpha(E_S) \subseteq \mathfrak{m}_S \mathfrak{m}_T = (0).
$$
Therefore, by the universal property of push-out, we obtain the existence of a map $\tilde{\alpha}$ such that the diagram
\begin{equation}\label{PushoutDiagram}
\begin{tikzcd}
\m_T &                                                 &                                              \\
     & E_R \arrow[lu, dashed, "\widetilde\alpha"]                         & E_T \arrow[l,"\iota_T",swap] \arrow[llu, "0"', bend right] \\
     & E_S \arrow[u,"\iota_S"] \arrow[luu, "\alpha", bend left] & k \arrow[l,"\eta_S"] \arrow[u,"\eta_T", swap]                       
\end{tikzcd} 
\end{equation}
commutes. The map $\widetilde\alpha$ is also obtained by the universal mapping property of quotient modules via the commutative diagram
\begin{equation}\label{eq20250129c}
\begin{split}
\xymatrix{
0\ar[r]&k\ar[rr]^{\left(\begin{smallmatrix}\eta_S\\ \eta_T\end{smallmatrix}\right)}\ar[d]&&E_S \oplus E_T\ar[d]^{\left(\begin{smallmatrix}\alpha& 0\end{smallmatrix}\right)}\ar[r]&E_R\ar[r]\ar[d]^{\widetilde\alpha}& 0\\
&0\ar[rr]&&\m_T \ar[r]&\m_T&
}
\end{split}
\end{equation}
Therefore, the map $\Hom_R(E_S, \mathfrak{m}_T) \xra{\zeta}  \Hom_R(E, \mathfrak{m}_T)$ given by $\zeta(\alpha)=\widetilde{\alpha}$ is $R$-linear. 
We also have a map
$$
\Hom_R(E, \mathfrak{m}_T) \xra{\vartheta} \Hom_S(E_S, \mathfrak{m}_T)
$$
defined by $\vartheta(\beta)=\beta \iota_S$, which is a left inverse for $\zeta$, that is,
$\widetilde{\alpha} \iota_S = \alpha$ for all $\alpha \in \Hom_R(E_S, \mathfrak{m}_T)$. This can be verified by examining the bottom-left triangle of diagram~\eqref{PushoutDiagram}. 
Hence, $\Hom_R(E_S, \mathfrak{m}_T)$ is a direct summand of $\Hom_R(E, \mathfrak{m}_T)$. Finally, as $R$-modules, $E_S$ and $\mathfrak{m}_T$ are annihilated by $\mathfrak{m}_T$ and $\mathfrak{m}_S$, respectively. Thus, $\Hom_R(E_S, \mathfrak{m}_T)$ is annihilated by $\mathfrak{m}_R$. Since $R$ is artinian, it follows that $\Hom_R(E_S, \mathfrak{m}_T)\neq 0$. Now, the assertion follows from~\ref{conv20250129a}.

That $\Hom_R(E_S,\m_T)$ is killed by $\m_R$ means that the image of each map $\alpha\in\Hom_R(E_S,\m_T)$ is contained in $\m_T\cap \soc(R)$.
Since $\soc(R)\cong\soc(S)\oplus\soc(T)$, $\m_T\cap\soc(R)\cong k^{\type(T)}$.
Hence $\Hom_R(E_S,\m_T)\cong\Hom_R(E_S,k^{\type(T)})\cong k^{\type(S)\type(T)}$, which gives the final assertion.
\end{proof}

\begin{lem}\label{lem20250129v}
Consider the setting of~\ref{conv20250129a}. If the ring $T$ is not Gorenstein, then $\Hom_R(E_T,\m_T)$ is a direct summand of $\Hom_R(E_R,\m_T)$, and 
\begin{equation}\label{eq20250129f}
\Hom_R(E_R,\m_T)\isom\Hom_R(E_S,\m_T)\oplus\Hom_R(E_T,\m_T)\cong\Hom_R(E_S,\m_T)\oplus\Hom_T(E_T,\m_T).
\end{equation}
\end{lem}
\begin{proof} Throughout the proof, we will write $E$ for $E_R$. 
If $S$ is a field, the statement is trivial.
Hence, assume that $S$ is not a field.
Since $T$ is artinian and not Gorenstein, we know that $E_T$ cannot be embedded in $T$.
Thus, for any $\alpha \in \Hom_T(E_T, T)$, by Lemma~\ref{lem20250129a}, we must have $\alpha \eta_T = 0$. Therefore, the outer square of the diagram
    \begin{equation}\label{PushoutDiagram2}
\begin{tikzcd}
\m_T &                                                 &                                              \\
     & E_R \arrow[lu, dashed, "\widetilde\alpha"]                         & E_T \arrow[l,"\iota_T",swap] \arrow[llu, "\alpha"', bend right] \\
     & E_S \arrow[u,"\iota_S"] \arrow[luu, "0", bend left] & k \arrow[l,"\eta_S"] \arrow[u,"\eta_T", swap]                       
\end{tikzcd} 
\end{equation}
commutes, giving rise to $\tilde{\alpha}$ by the universal property of the push-out. Similar to Lemma~\ref{lem20250129d}, the left inverse of $\alpha \mapsto \widetilde{\alpha}$ is precomposition by $\iota_T$, so that $\Hom_R(E_T, \mathfrak{m}_T)$ is a direct summand of $\Hom_R(E_R, \mathfrak{m}_T)$. To identify its complement, we observe that the image of $\Hom_R(E_R, \mathfrak{m}_T)$ intersects trivially with the direct summand $\Hom_R(E_S, \mathfrak{m}_T)$, identified by Lemma~\ref{lem20250129d} (where $S$ is not a field, as that case was trivial). Applying $\Hom_R(-, \mathfrak{m}_T)$ to the exact sequence
$$
0 \to k \to E_S \oplus E_T \to E_R \to 0
$$
by exactness, we obtain the inequality
$$
\ell_R(\Hom_R(E_R, \mathfrak{m}_T)) \leq \ell_R(\Hom_R(E_S \oplus E_T, \mathfrak{m}_T)).
$$
Also, note that both $\Hom_R(E_S, \mathfrak{m}_T)$ and $\Hom_R(E_T, \mathfrak{m}_T)$ are direct summands of, and hence, inject into $\Hom_R(E_R, \mathfrak{m}_T)$, and furthermore, their intersection is zero. Thus, the isomorphism~\eqref{eq20250129f} follows.
\end{proof}

The following result follows directly from~\eqref{eq20250129d} and Lemmas~\ref{lem20250129d} and~\ref{lem20250129v}.

\begin{thm}\label{42new}
Consider the setting of~\ref{conv20250129a}.
If $S$ and $T$ are not Gorenstein, then $$E_R^*\cong \Hom_R(E_S\oplus E_T,\m_S\oplus\m_T)\cong \Hom_S(E_S,S)\oplus \Hom_T(E_T,T)\oplus k^{\oplus2\type(S)\type(T)}$$
In particular, $E_S^*$ and $E_T^*$ are annihilated by $\m_S$ and $\m_T$, respectively, if and only if $E_R^*$ is annihilated by $\m_R$. 
\end{thm}
\begin{proof}
Applying Lemma~\ref{lem20250129v} gives the first isomorphism, which implies 
    \[E_R^*\cong\Hom_R(E_S,\m_S)\oplus\Hom_R(E_S,\m_T)\oplus\Hom_R(E_T,\m_S)\oplus\Hom_R(E_T,\m_T).\]
    
Since $S,T$ are homomorphic images of $R$ and neither are Gorenstein, $\Hom_R(S, \m_S)=\Hom_S(S, \m_S)=\Hom_S(E_S,S)$ and similarly for $T$.  
This fact, combined with applying Lemma~\ref{lem20250129d} to the terms $\Hom_R(E_S,\m_T)$ and $\Hom_R(E_T,\m_S)$, gives the second desired isomorphism.

Recalling that $\m_R=\m_S\oplus\m_T$, $\Hom_S(E_S,S)$ and $\Hom_T(E_T,T)$ are annihilated by $\m_S$ and $\m_T$, respectively, if and only if both of them are annihilated by $\m_R$, if and only if  $$E_R^*\cong \Hom_S(E_S,S) \oplus \Hom_T(E_T,T) \oplus k^{2\type(S)\type(T)}$$ is annihilated by $\m_R$.  
 
    

\end{proof}

\section{Properties of artinian rings with $E^*$ being a $k$-vector space}\label{sec20240911d}

In Proposition~\ref{siz}, we showed that, over an artinian local ring $R$, if $E^*$ is a $k$-vector space, then its vector space dimension equals $\left(\type(R)\right)^2$. In this section, we investigate more consequences of $E^*$ being a $k$-vector space.

\begin{para}\label{para20240911b}
The \emph{trace} of an $R$-module $M$, denoted $\tr_R(M)$, is defined to be
$$
\tr_R(M)=\sum_{f\in M^*}\im(f).
$$
Following~\cite{HHS}, a Cohen–Macaulay local ring $R$ with a canonical module $\omega_R$ is called \emph{nearly Gorenstein} if $\m\subseteq \tr_R(\omega_R)$.
\end{para}
 
Lyle and Maitra~\cite[Theorem 4.10]{Lyle}, in their work on Tachikawa Conjecture, prove that a local Cohen-Macaulay numerical semigroup ring $R$ with a canonical module $\omega_R$ and of minimal multiplicity is nearly Gorenstein if $\m \Ext^{i>0}_R(\omega_R,R)=0$. Our first goal in this section is to prove part~\eqref{int13a} of Theorem~\ref{int13}, which classifies the artinian non-Gorenstein nearly Gorenstein local rings under the assumption that $\m E^*=0$, i.e., $\m \Ext^0(\omega_R,R)=0$. For the proof, we build a systematic approach to the condition $\fm E^*=0$ starting from the following discussion.

\begin{para}\label{cor20230918b}
Let $M$ be n module over a local ring $R$. The following statements are well-known and straightforward to verify.
\begin{enumerate}[\rm(a)]
\item\label{cor20230918b1}
$\ann_R(M)=\ann_R(M^\vee)$.
\item\label{cor20230918b2}
A cyclic $R$-module is annihilated by $\fm$ if and only if it is isomorphic to $k$.
\item\label{cor20230918b3}
Using the evaluation map $f$ from~\ref{para20240915v}, the $R$-module $M^*$ is isomorphic to a direct summand of the $R$-module $M^{***}$.
\item\label{cor20230918b4}
If $M$ does not have a non-zero free summand, then we have the isomorphism $\Hom_R(M,\fm)\cong \Hom_R(M,R)$.
\item\label{lem20230918a}
	Let $\theta\colon k\to E$ be the inclusion map, and let $f\colon k\to E$ be an injective homomorphism. Then, $f=a\theta$ for some $a\in R\setminus \fm$.
\end{enumerate}
\end{para}

\begin{para}\label{tinytr} Any nonzero $R$-module $M$ surjects onto $R/\m$, and hence $\soc(R)=\tr_R(k)\subseteq \tr_R(M)$. We now record here that $\m M^*=0$ holds if and only if $\tr_R(M)=\soc(R)$. Indeed, for any $f\in M^*$, $\m f=0 \iff \m \operatorname{Im}(f)=0 \iff \operatorname{Im}(f)\subseteq \soc(R)$, hence we are done by definition of trace, and the  inclusion noted previously.  
\end{para}

\begin{lem}\label{prop20230918d}
The isomorphism $\Ext^1_R(\Tr E,R)^\vee\cong R/ \tr_R(E)$ holds. In particular, 
$\ann \Ext^1_R(\Tr E,R)=\tr_R(E)$.
\end{lem}

\begin{proof} 
Consider the canonical maps
\begin{align*}
&\alpha\colon E^*\otimes_R E\to R;& \alpha(f\otimes_R x)&=f(x)\\ 
&\beta\colon E^*\otimes_R E\to \Hom_R(E,E);& \beta(f\otimes_R x)&=\left(y\mapsto f(y)x\right)\\
&\gamma\colon R\to \Hom_R(E,E);& \gamma(a)&=(x\mapsto ax).
\end{align*}
Then, $\gamma$ is an isomorphism. As $f(y)x=f(xy)=f(x)y$ for all $x,y\in E$ and $f\in E^*$, we have $\gamma\alpha=\beta$. Thus, $\coker(\alpha)\cong \coker(\beta)$. Also, $\im(\alpha)=\tr_R(E)$ and we have
$$
\coker(\beta)=\Tor_1^R(\Tr E,E)\cong \Ext^1_R(\Tr E,R)^\vee.
$$
This fact along with~\ref{cor20230918b} (a) completes the proof. 
\end{proof}

\begin{lem}\label{prop20230918e}
The following statements are equivalent:
\begin{enumerate}[\rm(i)]
\item\label{prop20230918e1}
$R$ is nearly Gorenstein;
\item\label{prop20230918e2}
$\fm\Ext_R^1(\Tr E,R)=0$;
\item\label{prop20230918e3}
$\Ext_R^1(\Tr E,R)\cong k$.
\end{enumerate}
\end{lem}

\begin{proof}
The assertion follows from Lemma~\ref{prop20230918d} and~\ref{cor20230918b} parts~\eqref{cor20230918b1} and~\eqref{cor20230918b2}.
\end{proof}

The previous observations enable us to recover the following from~\cite{trcan}.

\begin{cor}
Assume that $R$ is a $d$-dimensional Cohen-Macaulay local ring with a canonical module $\omega_R$. Then, $\tr_R(\omega_R)$ is contained in $\ann \Ext^{d+1}_R(\Tr\omega_R,R)$.
\end{cor}

\begin{proof}
By~\cite[Lemma 11.41]{lw}, for any element $t\in \tr_R(\omega_R)$, the multiplication map $R\xra{t} R$ factors through a direct sum of copies of $\omega_R$. Hence, the induced multiplication map $\Ext^{d+1}_R(\Tr\omega_R,R)\xra{t} \Ext^{d+1}_R(\Tr\omega_R,R)$ factors through a direct sum of copies of $\Ext^{d+1}_R(\Tr\omega_R,\omega_R)$. However, this Ext module vanishes because $\omega_R$ has injective dimension $d$. Therefore, $t\in \ann \Ext^{d+1}_R(\Tr\omega_R,R)$, as desired.
\end{proof}

\begin{lem}\label{prop20230918c}
The following statements hold true:
\begin{enumerate}[\rm(a)]
\item\label{prop20230918c1}
$\fm E^*=0$ if and only if $\fm E^{**}=0$.
\item\label{prop20230918c2}
$\fm(E/k)=0$ if and only if $\fm^2=0$.
\end{enumerate}
\end{lem}

\begin{proof}
\eqref{prop20230918c1}
The ``only if'' part is obvious. The ``if'' part follows by
applying~\ref{cor20230918b}~\eqref{cor20230918b3} to $M=E$.

\eqref{prop20230918c2}
Applying the Matlis dual functor to the exact sequence $0\to \fm\to R\to k\to 0$,
we have an exact sequence $0\to k\to E\to \fm^\vee\to 0$. It follows from~\ref{cor20230918b}~\eqref{lem20230918a} that
$\fm^\vee\cong E/k$. The assertion now follows from~\ref{cor20230918b}~\eqref{cor20230918b1}.
\end{proof}

We are now ready to prove Theorem~\ref{int13}~\eqref{int13a}. Note that the assumption that $R$ is not Gorenstein is neccesary by Example~\ref{ex20240917d}.\vspace{2mm}

\noindent \emph{Proof of Theorem~\ref{int13}~\eqref{int13a}.}
Assume that  $\m^2=0$. It follows from~\cite[Example 2.6]{IT} that $\soc(R)$ annihilates $\Ext_R^1(\Tr E,R)$. Hence, by Lemma~\ref{prop20230918d} we have $\soc(R) \subseteq \tr_R(E)$. Therefore, by assumption we have $\fm\subseteq \tr_R(E)$, i.e., $R$ is nearly Gorenstein. Furthermore, by our assumption, $\fm \Hom_R(E, \fm) = 0$. Note that $E$ is indecomposable and non-free because $R$ is not Gorenstein. Thus, $\fm E^* = 0$ by~\ref{cor20230918b}~\eqref{cor20230918b4}.

For the converse, assume that $R$ is nearly Gorenstein and $\m E^*=0$. By Lemma~\ref{prop20230918e}, the exact sequence
$0 \to \Ext_R^1(\Tr E, R) \to E \to E^{**}$ induces an exact sequence $0 \to k \to E \to E^{**}$. Hence, it follows from~\ref{cor20230918b}~\eqref{lem20230918a} that there is an injective homomorphism \(E/k \to E^{**}\). Therefore, it follows from Lemma~\ref{prop20230918c}~\eqref{prop20230918c1} and~\eqref{prop20230918c2} that $\m^2=0$, as desired.
\qed\vspace{2mm}

We now aim for the proof of Theorem~\ref{int13}~\eqref{int13b} that uses the following lemma.

\begin{lem}\label{1}
Assume that $R$ is artinian and choose an artinian Gorenstein local ring $(S,\n,k)$ and an ideal $I$ of $S$ such that $R\cong S/I$.
Then, there exist isomorphisms
\begin{align*}
E\cong (0:_SI)&&&
E^*\cong\Hom_S((0:_SI),S/I)
\end{align*}
and a short exact sequence
\begin{equation}\label{eq20211116b}
\!\!0\to S/(I:_S(0:_SI))\xra{f}\Hom_S((0:_SI),S/I)\to\Ext_S^1(S/(0:_SI),S/I)\to 0
\end{equation}
of $S$-modules, where $f(\overline a)(x)=\overline{ax}$ for $\overline a\in S/(I:_S(0:_SI))$ and $x\in (0:_SI)$.
\end{lem}

\begin{proof}
Since $S$ is artinian and Gorenstein, $E_S(k)\cong S$. Therefore, we have the isomorphisms $E\cong \Hom_S(S/I,E_S(k))\cong \Hom_S(S/I, S)\cong (0:_SI)$. This implies that $E^* \cong \Hom_{S/I}((0:_SI), S/I)\cong \Hom_S((0:_SI), S/I)$. Applying the functor $\Hom_S(-,S/I)$ to the short exact sequence $0 \to (0:_SI) \to S \to S/(0:_SI) \to 0$, we obtain an exact sequence
\small{
$$ 
0 \to \Hom_S(S/(0:_SI), S/I) \xra{g} S/I \xra{e} \Hom_S((0:_SI), S/I) \to \Ext_S^1(S/(0:_SI), S/I) \to 0
$$
}\normalsize
of $S$-modules in which $g(\alpha) = \alpha(\overline{1})$ for $\alpha \in \Hom_S(S/(0:_SI), S/I)$ and $e(\overline{a})(x) = \overline{ax}$ for $\overline{a} \in S/I$ and $x \in (0:_SI)$.
There is an isomorphism
$$
\left(I:_S(0:_SI)\right)/I = \left(0:_{S/I}(0:_SI)\right) \xra{h} \Hom_S(S/(0:_SI), S/I) 
$$
given by $h(\overline{b})(\overline{c}) = \overline{bc}$ for $\overline{b} \in \left(I:_S(0:_SI)\right)/I$ and $\overline{c} \in S/(0:_SI)$. 
Hence, we obtain an exact sequence
$$
0\to \left(I:_S(0:_SI)\right)/I  \xra{gh} S/I \xra{e} \Hom_S((0:_SI), S/I) \to \Ext_S^1(S/(0:_SI), S/I) \to 0
$$
where $gh(\overline{b}) = \overline{b}$ for $\overline{b} \in \left(I:_S(0:_SI)\right)/I$. 
This induces the exact sequence~\eqref{eq20211116b} with the desired properties.
\end{proof}

\noindent \emph{Proof of Theorem~\ref{int13}~\eqref{int13b}.} 
By Lemma~\ref{1}, the condition $\m E^*=0$ translates to $\n\Hom_S((0:_SI),S/I)=0$. Hence, it follows from the short exact sequence~\eqref{eq20211116b} that $\n(S/(I:_S(0:_SI)))=0$.
Therefore, $\n\subseteq (I:_S(0:_SI))$, which means that the containment $\n(0:_SI)\subseteq I$ holds.
\qed\vspace{2mm}

Using Theorem~\ref{int13}~\eqref{int13b}, the next example shows that artinian Golod rings do not necessarily satisfy the equality $\m E^*=0$.

\begin{ex}
	Let $R=k[x,y]/(x^4,x^2y^2,y^4)$. Note that $R\cong S/I$, where $S=k[x,y]/(x^4,y^4)$ is an artinian Gorenstein local ring with the maximal ideal $\fn=(x,y)$ and $I=(x^2y^2)$. Then, $(0:_SI)=(x^2,y^2)$ and we have $\fn(0:_SI)=\fn^3$, which is not contained in
$I$. Therefore, by Theorem~\ref{int13}~\eqref{int13b} we have $\m E^*\neq 0$.
\end{ex}

The next example shows that the converse of Theorem~\ref{int13}~\eqref{int13b} may not hold.

\begin{ex}\label{ex20240919a}
Let $R=k[x]/(x^2)$. Note that $R\cong S/I$, where $S=k[x]/(x^3)$ is the artinian Gorenstein local ring with $\fn=(x)/(x^3)$ and $I=(x^2)/(x^3)=\soc(S)$.
Clearly, $(0:_SI)=\fn$ and thus, $\fn(0:_SI)=\fn^2=I$. However, $\fm E^*=\fm\neq 0$.

If we choose $S=K[x]/(x^4)$ and $I=\soc(S)$, then $I=(x^2)/(x^4)$ and $R\cong S/I$.
In this case, $(0:_SI)=(x^2)$ and hence, $\fn(0:_SI)=(x^3) \subsetneqq I$. Moreover, again $\fm E^*\neq 0$.
\end{ex}

In light of Example~\ref{ex20240919a}, it is natural to ask the following question.

\begin{question}
Assume that $R$ is artinian. Under what condition does the converse of Theorem~\ref{int13}~\eqref{int13b} hold?
More precisely, if $(S,\n,k)$ is an artinian Gorenstein local ring and $I$ is an ideal of $S$ such that $R\cong S/I$, then when does $\n(0:_S I)\subseteq I$ imply that $\m E^* =0$?
\end{question}

\section{The $1$-dimensional case}

To investigate Question~\ref{question20240911a} in the $1$-dimensional case, we present the following variant of~\ref{cor20230918a}. For the definition and fundamental properties of Ulrich modules, we refer to \cite{BHU, Ulrich} and~\cite{Goto}.

\begin{discussion}\label{para20241014a}
	If \( R \) is a non-Gorenstein Cohen-Macaulay local ring of minimal multiplicity with \( \dim(R) \leq 1 \) and canonical module \( \omega_R \), then \( \omega_R^* \) is Ulrich (hence, a \( k \)-vector space when \( \dim(R) = 0 \)). \end{discussion}
	
\begin{proof}	To see this, note that \( R \) is not a direct summand of \( \omega_R \) and hence, it is not a direct summand of \( \omega_R^* \) by~\cite[Lemma 5.14]{dg}. Also, \( \omega_R^* \) is a second syzygy module and thus, it is a syzygy of a maximal Cohen-Macaulay \( R \)-module. Therefore, by~\cite[Proposition 1.6]{ukt}, we get that \( \omega_R^* \) is Ulrich. 
\end{proof}

Theorem~\ref{int13'} is a supped up version of our discussion in~\ref{para20241014a} to a higher power of $\fm$. In other words, if $\fm$ is Ulrich, then $R$ has minimal multiplicity, which has already been discussed in~\ref{para20241014a}. For the proof of Theorem~\ref{int13'} we need the following lemma that follow by the same argument as in the proof of~\cite[Lemma 2.5]{y}. 

\begin{lem}\label{lem20241207a}
Let $M$ be an $R$-module, and let $x \in \fm\setminus \fm^2$ be an $R$-regular element whose initial form in the associate graded ring $\gr_\fm(R)$ is regular on $\gr_\fm(M)$. Then, for every integer $n\geq1$, there is an isomorphism
$$
\fm^nM/x\fm^nM\cong \left(\fm^{n-1}M/\fm^nM\right)\oplus \left(\fm^nM/x\fm^{n-1}M\right).
$$
\end{lem}

\noindent  \emph{Proof of Theorem~\ref{int13'}.} 
Replacing $R$ with $S:=R[X]_{\fm[X]}$, one can assume without loss of generality that $k$ is infinite. Let us for example check the depth condition. Namely, the key point is that $S / \fm S = (R[X]_{\fm R[X]} / \fm R[X]_{\fm R[X]}) \cong (R[X] / \fm R[X])_{\fm R[X]} $ is a field. Then $\fn := \fm S$ is its unique maximal ideal. Now,  

\[ \mathrm{gr}_\fm(\omega^\ast) \otimes_R S \cong \bigoplus_n \frac{(\fm^n \omega^\ast) \otimes_R S}{(\fm^{n+1} \omega^\ast) \otimes_R S} \cong \bigoplus_n \frac{\mathfrak{n}^n \omega_S^\ast}{\mathfrak{n}^{n+1} \omega_S^\ast} = \mathrm{gr}_{\mathfrak{n}}(\omega_S^\ast)\]
where $\omega_S$ is the canonical module of $S$ and so, $\depth(\gr_{\fn}(\omega_S^*))=1$.

We now begin by assuming that for some integer $n\geq 2$ the ideal $\fm^n$ is Ulrich and work towards proving that $\fm^{n-1}\omega^*$ is also Ulrich. Near the end, we will specialize to the case $n=2$, which gives us the assertion of Theorem~\ref{int13'}.

Consider an integer $n \geq 2$. Since $k$ is infinite, the condition of having a principal reduction of $\fm$ is a general condition. Similarly, having $x \in \fm \setminus \fm^2$ whose initial form is regular on $\gr_\fm(\omega^*)$ is also a general condition.
The intersection of general conditions is again general because the intersection of non-empty Zariski open sets in $\fm / \fm^2$ is also a non-empty open set. Hence, we can find $x \in \fm \setminus \fm^2$ such that $\fm^{n+1} = x\fm^n$ and whose initial form in $\gr_\fm(R)$ is regular on $\gr_\fm(\omega^*)$.
To show that $N := \fm^{n-1}\omega^*$ is Ulrich, it suffices to show that $N/xN$ is isomorphic to a direct sum of copies of $k$. Since $n-1 \geq 1$, by Lemma~\ref{lem20241207a}, we get that $N/xN$ is isomorphic to the direct sum of $\fm^{n-2}\omega^* / \fm^{n-1}\omega^*$ and $\fm^{n-1}\omega^* / x\fm^{n-2}\omega^*$.
Now, $\fm^{n-2}\omega^* / \fm^{n-1}\omega^*$ is already annihilated by $\fm$, so we only need to show that $\fm^{n-1}\omega^* / x\fm^{n-2}\omega^*$ is also annihilated by $\fm$. When $n = 2$, this means we need to show that $\fm\omega^* / x\omega^*$ is annihilated by $\fm$. Applying $\Hom_R(\omega, -)$ to the exact sequence
$$
0 \to R \xra{x} R \to R/xR \to 0
$$
we see that $\omega^* / x\omega^*$ is isomorphic to a submodule of $\Hom_R(\omega, R/xR)$. Since $R/xR$ is not Gorenstein, we have
$$
\Hom_R(\omega, R/xR) \cong \Hom_{R/xR}(\omega / x\omega, R/xR)\cong \Hom_{R/xR}(\omega_{R/xR}, \fm/Rx)
$$
which then embeds into $\fm F$, for some free $R/xR$-module $F$.
Thus, the submodule $\fm\omega^* / x\omega^*$ of $\omega^* / x\omega^*$ embeds into $\fm(\fm F) = \fm^2 F$. Since $F$ is a free $R/xR$-module and $\fm^3 = x\fm^2$, it follows that $\fm^2 F$ is annihilated by $\fm$. Hence, $\fm\omega^* / x\omega^*$ is also annihilated by $\fm$, which completes the proof.
\qed

\begin{para}
Far-flung Gorenstein rings of dimension 1 are defined by the condition that the trace ideal $\text{tr}(\omega_R)$ of the canonical module equals the conductor ideal; see~\cite{HHS2}.   
\end{para}

\begin{cor}
Let $R$ be a complete  Far-flung Gorenstein domain with the canonical module $\omega_R$. Then, $\fm\omega_R^*$ is Ulrich when $\fm^2$ is Ulrich. 
\end{cor}
 
\begin{proof}
We have $\omega_R^* = \omega_R^{-1} = \overline{R}$; see~\cite[A.4]{HHS2}. This implies $ \text{gr}_\mathfrak{m}(\omega_R^*) = \text{gr}_\mathfrak{m}(\overline{R}) = \text{gr}_t^n(\overline{R}), $ where \( n \) satisfies \( \mathfrak{m} \overline{R} = (t^n) \) and \( t \) is a uniformizer for \( \overline{R} \). Consequently, \( \text{gr}_\mathfrak{m}(\omega_R^*) \) is the \( n \)-th Veronese subring of \( k[t] \), which is Cohen-Macaulay and thus of positive depth, so  the desired claim follows from Theorem \ref{int13'}. 
 \end{proof}
 
 
\begin{para}\label{para20250129g}
A natural question to ask is: when do 1-dimensional Cohen-Macaulay local rings of minimal multiplicity satisfy the depth condition $\depth(\gr_\fm(\omega_R^*))=1$?

The Valabrega-Valla criterion shows that the initial form of an element $a$ is regular on $\gr_\fm(\omega_R^*)$ if and only if 
\begin{equation}\label{eq20250129e}
\fm^i(\omega_R^*) \cap a(\omega_R^*) = \fm^{i-1}a(\omega_R^*)
\end{equation}
for all $i>0$, where $a\in\fm$ is  such that $(a)$ is a minimal reduction of $\fm$. Hence, the answer to the above question is affirmative. Indeed, if $R$ has minimal multiplicity, then $\fm^2=a\fm$, so the equality~\eqref{eq20250129e} holds for all
$i>0$. This implies that $\gr_\fm(\omega_R^*)$ has positive depth. In fact, since $\omega_R^*$ is in $\Omega \text{CM}(R)$ and $R$ has minimal multiplicity, it is Ulrich, so that the conclusion also follows from~\cite[Corollary 1.6]{BHU}.
\end{para} 



\begin{para}\label{6.7}
Finally, one may ask: If $n$ is a positive integer such that $\fm^n$ is Ulrich, is $\fm^{n-1} \omega_R^*$ also Ulrich?

The answer is again negative. Let $R=k[\![t^7,t^{11},t^{13},t^{15}]\!]$. One can check that $\{t^{14},t^{18},t^{20},t^{22},t^{24},t^{26},t^{30}\}$ is a minimal generating set for $\m^2$. In particular, since multiplicity of $R$ is $7$, the ideal $\m^2$ is Ulrich. By standard numerical semigroup methods (i.e. via calculating the psuedo-Frobenius numbers) we can calculate a canonical ideal for $R$ to be $I=(t^7,t^{11},t^{13})$. Then, $I^* \cong (t^7:_RI)=(t^7,t^{24},t^{26},t^{30})$. Thus, $\m I^*=(t^{14},t^{18},t^{20},t^{22})$. In particular, $\m I^*$ is not Ulrich.
\end{para}

We conclude the paper with the following example relating to Theorem \ref{int13'}.

\begin{ex}Let $R=k[[t^5,t^6,t^7]]$. Then
$\fm^2=(t^{10},t^{11},t^{12},t^{13},t^{14})$ is Ulrich. We have $\omega_R=(1,t)$ and
$N:=\omega_R^*=(t^5,t^6,t^{14}).$ As $\fm N=(t^{10},t^{11},t^{12},t^{13})$ has less than $5$ generators, $\fm N$ is not Ulrich. The
multiplication map $N/ \fm N \to \fm N/\fm^2N$ by $t^5$ sends $t^{14}$ to $0$. Hence depth
$\gr_\fm(N)=0$.\end{ex}
\section*{Acknowledgments}

The authors are grateful to Justin Lyle, Kazuho Ozeki and Toshinori Kobayashi for reading an earlier version of this paper and providing numerous comments that improved the paper. Especially, we thank Toshinori Kobayashi for  \ref{tinytr} and Justin Lyle for the example used in \ref{6.7}.

\end{document}